\documentclass[11pt]{article}

\usepackage{amsmath,amsthm,amssymb,mathrsfs}
\newtheorem{df}{Definition}[section]
\newtheorem{lem}[df]{Lemma}

\newtheorem{pro}[df]{Proposition}
\newtheorem{thm}[df]{Theorem}

 \usepackage{fancyhdr}
\usepackage{atbegshi}
\usepackage{dsfont}
\usepackage{xcolor}

\newcommand{\R}{\mathbb{R}}
\author{Anouar Jeddi}

\title{Probabilistic large deviations approach to obstacle Hamilton-Jacobi equations}

\begin{document}
\maketitle
\begin{abstract}
In this paper, we derive a non-classical variational representation of a viscosity solution of a Hamilton–Jacobi equation with obstacle, in which the Hamiltonian may take infinite values. We study a nonlinear selection–mutation model with an exponential-decay mutation kernel, using a probabilistic Feynman–Kac representation of a pure jump process. To handle the singular character of this process, we rely on a non-standard tool from large deviations theory, which allows us to establish a Varadhan-type lemma. This, in turn, yields the variational representation of the viscosity solution.
\end{abstract}
\maketitle

\section{Introduction}
Our main goal is to derive, using a probabilistic method, the variational representation of a viscosity solution of the following constrained Hamilton–Jacobi equation with obstacle: for all $(t,x)\in(0,+\infty)\times  \mathbb{R}^d,$
\begin{equation}
    \label{eq:HJ3}
    \begin{cases}
      \min\Big( \partial_tu(t,x)-R(x,I)-\int_{\mathbb{R}^d}G(y)(e^{\nabla u(t,x).y}-1) dy, 1-\vert \nabla u(t,x)\vert \Big)=0,\\
      \max_{x\in \R^d}u(t,x)=0,~~\forall t>0,\\
      u(0,.)=u^0(.),
      \end{cases}
    \end{equation}
where $I$ is a function of bounded variation. \\

Equation \eqref{eq:HJ3} is therefore not a classical object: its Hamiltonian may take infinite values and is discontinuous in time, so that it falls outside the scope of the standard theory of viscosity solutions, and its well-posedness cannot be deduced from classical results. For background on the definition of viscosity solutions for Hamiltonians that are merely measurable, or discontinuous, in time, we refer to~\cite{Ishi,LP}, and to~\cite{G} for the classical theory of viscosity solutions. To overcome this difficulty, we develop a probabilistic method that not only allows us to derive this equation directly, but also yields a variational representation of its viscosity solution. This representation turns out to be a crucial tool for the subsequent analysis, in particular for establishing uniqueness (see \cite{CL}).\\

In the field of eco-evolutionary dynamics a Hamilton-Jacobi approach was emerged to study a concentration phenomenon in populations structured with traits and evolving under selection, mutation and competition (see, \cite{OPSB}). This approach considers mutations of order $\varepsilon$ where $\varepsilon$ is a small positive parameter, and shows that in large time the population density concentrates around some Dirac masses described by a viscosity solution of a Hamilton-Jacobi equation (see, for instance \cite{PB,BMP}). In \cite{BMP} an integro-differential model was considered assuming, that in a regime of small mutations and large time that the density $n_\varepsilon(t,x)$ at time $t$ and  trait $x$, satisfies that
\begin{equation}
\label{eq:model}
\begin{cases}
\varepsilon \partial_t n_\varepsilon(t,x)=R(x,I_\varepsilon(t))n_\varepsilon(t,x)+\int_{\mathbb{R}^d}G(y)(n_\varepsilon(t,x+\varepsilon y)-n_\varepsilon(t,x))dy,\\
I_\varepsilon(t)=\int_{\mathbb{R}^d}n_\varepsilon(t,y)dy,\\
n_\varepsilon(0,.)=n_\varepsilon^0
(.),
\end{cases} 
\end{equation}
where $R(,)$ is the growth rate which depends on traits and the total population size $I_\varepsilon$ and $G$ is a probability kernel. The analysis of \eqref{eq:model} was carried out in \cite{BMP} for a class of mutation kernels with super-exponential decay that making the analysis classical. However, kernels with slower decay are more interesting in terms of mathematical modeling because they allow mutations of big size to occur at a high rate\\
The analysis of kernels with slow decay was addressed only in a few different settings than \eqref{eq:model}, for instance singular kernels in \cite{MM, Sepideh} and fat-tailed kernels in \cite{BGHP}.\\

Recently, the author in \cite{J} studied kernels with exponential decay and derived the Hamilton–Jacobi equation \eqref{eq:HJ3} featuring an obstacle. This was achieved by proving that the Hopf-Cole transform $u_\varepsilon(t,x)=\varepsilon \log{n_\varepsilon(t,x)}$ converges along a subsequence of $\varepsilon$ to a continuous function $u,$ which is a viscosity solution of \eqref{eq:HJ3}.
   The derivation of this equation is very technical due to the non-finite nature of the Hamiltonian (see, \cite{J} for more details).\\
   
   In this work, we present a more direct method to derive \eqref{eq:HJ3} and establish a non trivial dynamic programming principle for the viscosity solution, based on a probabilistic approach.  In \cite{Freidlin,MF} Freidlin used a probabilistic approach to study the propagation phenomena in some reaction–diffusion equations, based on the Feynman–Kac representation and Varadhan's lemma. Moreover, in the context of population dynamics, this kind of approach has been used in \cite{NB} for a selection–mutation model with the Laplace operator, instead of a non-local operator, to study the concentration phenomena. The strategy of our approach is to first establish a Feynman-Kac formula based on a pure jump Markov process and use the Large deviation principle (LDP) combined with Varadhan's lemma to derive a variational representation of the viscosity solution of \eqref{eq:HJ3}.\\

The probabilistic representation of the integro-differential equation  \eqref{eq:model} is based on a pure Markov jump process
$
(X_t^\varepsilon)_{t\geq0},$
driven by a Poisson random measure with intensity \(ds\,G(dy)\). The next step is to establish a large deviation principle for the process \((X_t^\varepsilon)_{t\geq0}\).

Large deviation principles for pure jump Markov processes and Lévy processes are well understood under the strong Cramér condition, that is, under the assumption that exponential moments of all orders exist (see, for instance,~\cite{FK,FS}). This assumption, however, turns out to be too restrictive for our setting: since our kernel $G$ only exhibits exponential decay, the strong Cramér condition fails.
 Large deviation principles under the weaker assumption
\begin{equation}
  \label{eq:integrability assumption}
     \exists\, \alpha>0,\quad \forall \beta \in (-\alpha,\alpha),\quad \mathbb{E}\big(e^{\beta X^\varepsilon_t}\big)<\infty,\quad \forall t>0,
\end{equation}
have already been investigated for Lévy processes on the Skorokhod space $\mathbb{D}([0,+\infty),\mathbb{R}^d)$, although only with respect to weaker topologies than the usual $J_1$ topology (see, for instance,~\cite{LS,C,dA}).

In \cite{M}, an extended large deviation principle was established for compound Poisson processes under the weak Cramér assumption~\eqref{eq:integrability assumption}, in the space of functions of bounded variation endowed with the Borovkov metric (recalled below). Indeed, under this weaker assumption, the rate function is in general not \emph{good} (its sublevel sets fail to be compact), so that a classical LDP cannot hold, the extended LDP is, to our knowledge, the only available substitute in this setting. We recall below the definition of the Borovkov metric together with the notion of extended LDP in the space of functions of bounded variation.

This large deviation principle differs from the classical one at the level of the upper bound. More precisely, the extended large deviation principle established in \cite{M} only provides an upper bound for approximating sets rather than for arbitrary closed sets, and the corresponding rate function is not necessarily good. As a consequence, Varadhan's lemma cannot be derived by the classical argument, which crucially relies on the goodness of the rate function.

A natural way to circumvent this difficulty would be to establish exponential tightness directly with respect to the Borovkov metric. However, this turns out to be a delicate task in our framework, and we were not able to obtain it directly.

To overcome this obstacle, we adopt a different strategy. We first construct an approximating process $\overline{X}^\varepsilon$ associated with $X^\varepsilon$, for which exponential tightness can be established, not with respect to the Borovkov metric, but with respect to the (weaker) topology induced by the Lévy metric. We then combine this exponential tightness with the extended large deviation principle for $X^\varepsilon$ obtained under the Borovkov metric, in order to derive a version of Varadhan's lemma adapted to our setting.

Let us emphasize that this strategy is not classical: it relies on the interplay between two distinct, and comparatively weak, topologies, the Borovkov metric, for which the extended LDP holds, and the Lévy metric, for which exponential tightness of the approximating process can be obtained. To the best of our knowledge, such a combination has not been used before to derive a Varadhan-type lemma in this weak Cramér setting.

\subsection{Function spaces and LDP framework}

Let $T>0.$ We denote by $\mathbb{V}_T$ the space of all real-valued functions
$f:\mathbb{R}\rightarrow\mathbb{R}$ satisfying the following properties:
\begin{itemize}
    \item $f$ has bounded variation;
    \item $f$ is left-continuous;
    \item $f(t)=0$ for all $t\leq0$;
    \item $f(t)=f(T+)$ for every $t>T$.
\end{itemize}

Alongside the space $\mathbb{V}_T,$ we consider the class $\mathbb{V}^+_T\subset \mathbb{V}_T$ of nondecreasing functions $f\in \mathbb{V}_T.$

Every function $f\in \mathbb{V}_T^+$ admits the decomposition
\[
f=f^{+}-f^{-},
\]
where $f^{+},f^{-}\in \mathbb{V}_T^{+}$ are uniquely determined by the condition
\[
\operatorname{Var}(f)
=
\operatorname{Var}(f^{+})
+
\operatorname{Var}(f^{-}).
\]
We also introduce the product spaces
$$
\widetilde{\mathbb{V}}_T
=\{\widetilde{f}=(f^+,f^-), f^{+\atop -}\in \mathbb{V}_T^+ \},$$
We denote by $\mathbb{V}^a_T$ the space of absolutely continuous 
functions in $\mathbb{V}_T$, and by $\widetilde{\mathbb{V}}^a_T$ 
the space of absolutely continuous functions in $\widetilde{\mathbb{V}}_T$.\\
For every $f\in \mathbb{V}_T$ we define its graph 
\begin{equation}
\Gamma(f) = \left\{ (t, z) \in \R \times \R^d 
\;\middle|\; z \in [f(t^-), f(t)] \right\}.
\end{equation}
Let $f,g\in \mathbb{V}_T$ the Borovkov metric denoted  $\rho_B$ is defined by: 
\begin{equation}
  \rho_B(f, g) = d_H\bigl(\Gamma(f),\, \Gamma(g)\bigr),
\end{equation}
where $d_H$ is the Hausdorff distance in $R^{d+1},$  i,e  $\forall A,B\in\R^{d+1},$
$$
d_H(A, B) = \max\left( 
\sup_{a \in A} \inf_{b \in B} \|a - b\|,\; 
\sup_{b \in B} \inf_{a \in A} \|a - b\| 
\right).
$$
Let us now recall the definition of an extended LDP in $(\mathbb{V}_T,\rho_B)$.  Let $A\subset \mathbb{V}_T $ and $\eta>0,$ we define the $\eta$-neighborhood of $A$ in $(\mathbb{V}_T,\rho_B)$ as 
\begin{equation}
    \label{eq:approximation set}
    A_\eta=\{g \in \mathbb{V}_T / \inf_{g\in A}\rho_B(f,g)<\eta\},
\end{equation}
and, we define \begin{equation}
    A_+:=\lim_{\eta\to 0} A_\eta.
\end{equation}
We say that the process
$(X^\varepsilon)_{\varepsilon>0}$ satisfies an extended LDP in $(\mathbb{V}_T,\rho_B)$ with a lower semi-continuous rate function $J_T$ if: for all open subset $A\subset \mathbb{V}_T,$ we have that 
    \begin{equation}
        \limsup_{\varepsilon\to 0}\varepsilon\log{\mathbb{P}(X^\varepsilon\in A)}\leq -J_T(A_+):=-\lim_{\eta \to 0}\inf_{x\in \overline{A}_\eta}J_T(x),
    \end{equation}
    and \begin{equation}
        \liminf_{\varepsilon\to 0}\varepsilon\log{\mathbb{P}(X^\varepsilon\in A)}\geq-J_T(A):= -\inf_{x\in A}J_T(x).
    \end{equation}



We now recall the definition of the Lévy metric and the main results of \cite{M} that will be needed in the sequel.

To define the Lévy metric, we first associate  every
$\tilde f\in\tilde{\mathbb{V}}_T$ with
its canonical parametrization $
\hat f(v)
=
(\hat f_{0}(v),\hat f_{+}(v),\hat f_{-}(v)),~
v\in\mathbb R.$ We say that  $\hat f$ is a parametrization of $\tilde{f}$ if: 
\begin{itemize}
\item For all $v\in \R,$ 
 $\hat f(v) \in \Gamma_{\tilde{f}}.$ 
    \item For all  $v\in \R,$
$
\hat f(v)
=
(\hat f_{0}(v),\hat f_{+}(v),\hat f_{-}(v)),$
characterized by the identity
\begin{equation}
\hat f_{0}(v)
+
\hat f_{+}(v)
+
\hat f_{-}(v)
=
v.\end{equation}
\end{itemize}
The Lévy metric on $\widetilde{\mathbb{V}}_T$ is then defined by
\begin{equation}
\label{eq:lévy metric}
\widetilde{\rho}_{L}(\tilde f,\tilde g)
=
\sup_{v\in\mathbb R}
\left(
|\hat f_{0}(v)-\hat g_{0}(v)|
+
|\hat f_{+}(v)-\hat g_{+}(v)|
+
|\hat f_{-}(v)-\hat g_{-}(v)|
\right).
\end{equation}
The link between the Borovkov and Lévy metrics, established in Lemma~2.2 of~\cite{M}, is recalled below.
\begin{lem}
For any $\tilde{f}, \tilde{g} \in \widetilde{\mathbb{V}}_T$, the following inequalities hold:
\begin{equation}
\label{eq:metric comparaison}
\widetilde{\rho}_{L}(\tilde{f}, \tilde{f})
\;\ge\;
\widetilde{\rho}_{B}(\tilde{f}, \tilde{g})
\;\ge\;
\rho_B(f^+ - f^-,\, g^+ - g^-).
\end{equation}
\end{lem}
\begin{lem}
    In the same probability space we denote a process $\overline{X}^\varepsilon=(\overline{X}^\varepsilon_+,\overline{X}^\varepsilon_-)\in \widetilde{\mathbb{V}}_T^a$ such that fro all $\eta>0,$ we have 
    \begin{equation}
        \label{eq:expo app}
        \limsup_{\varepsilon\to 0}\varepsilon\log{\mathbb{P}(\widetilde{\rho}_L(\overline{X}^\varepsilon,\widetilde{X}^\varepsilon)>\eta)}=-\infty,
    \end{equation}
    where $\widetilde{X}^\varepsilon=(X^\varepsilon_+,X^\varepsilon_-).$
\end{lem}
Note that from \eqref{eq:metric comparaison} the approximation 
\eqref{eq:expo app} holds also for the Borovkov metric $\widetilde{\rho}_B$.
For the construction of the process $\overline{X}^\varepsilon$ 
and the proof of the exponential approximation, 
we refer to \cite{M}, Section 4.\\
The following proposition states the key properties 
of the approximated process $\overline{X}^\varepsilon$, 
as given in \cite{M}, Lemma~7.1.
\begin{pro}
    \begin{itemize}
    \item For all $\tilde{f}\in \mathbb{\widetilde{V}_T},$ we have 
    \begin{equation}
        \label{eq:lower ldp for app pr}
    \displaystyle\lim_{\delta\downarrow0}\limsup_{\varepsilon \rightarrow 0}\log
P(\overline{{X}}^\varepsilon \in B_L({\tilde f},{\delta}))\le-J(\tilde f).
    \end{equation}
        \item Let $M>0,$ there exists a compact set $\mathcal{K}_M\subset \widetilde{\mathbb{V}}_T$ such that
        \begin{equation}
\label{eq:expontighntess apro}
    \limsup_{\varepsilon\to 0}\varepsilon\log{\mathbb{P}(\overline{X}^\varepsilon\in \mathcal K_M^c)}\leq -M.
\end{equation}
    \end{itemize}
\end{pro}








This paper is structured as follows. Section 2 recalls some preliminary results. Section 3 is devoted to the Feynman-Kac representation. In Section 4, we derive our Varadhan's lemma, and Section 5 is dedicated to the link between the variational representation and the Hamilton-Jacobi equation \eqref{eq:HJ3}.
\subsection{Assumptions}
   \begin{enumerate}
        \item \label{item:13} There exist two positive constants $I_m$ and $I_M$ such that 
        \begin{equation}
            \label{eq:growth bound }
            \inf_{x\in \R^d} R(x,I_m)=0,~~~~~~\sup_{x\in \R^d}R(x,I_M)=0,
        \end{equation}
        and \begin{equation}
        \label{eq:intitial size population bound}
            I_m/2\leq I_\varepsilon(0)\leq 2I_M,~~\forall \varepsilon>0.
        \end{equation}
       \item \label{item:23} We assume that $R$ is Lipschitz-continuous functions, in space with Lipschitz norms $\Vert R \Vert_{\text{Lip}}$ and uniformly w.r.t $I\in [I_m/2,2I_M]$, and that there exists a positive constant $\overline{R},$ such that for any $x\in \R^d,$
       \begin{equation}
       \label{growth canditions}
  \sup_{I\in [I_m/2,2I_M]}\Vert R(.,I)\Vert_{W^{2,\infty}(\R^d)} \leq \overline{R} \quad.
       \end{equation}
      \item \label{item:33} The growth rate $R(.,.)$ is differentiable w.r.t $I$ for all $x\in \R^d,$ and
      there exist positive constant $A_1, A_2$ and $A_3$ such that for any $I,I_1,I_2\in \R^+$ and $x\in \R^d$
      \begin{equation}
      \label{eq:growth monotonie in I}
        \vert R(x,I_1)-R(x,I_2)\vert \leq A_1\vert I_1-I_2\vert,~~~-A_2\leq \partial_IR(x,I)\leq -A_3.
 \end{equation}
        
        \item \label{item:43} The kernel $G$ has an exponential decay, i.e., $G(x)=f(x)e^{-\vert x\vert },$  such that  $\int_{\mathbb{R}^d}G(y)dy=1,$
         $G$ is a positive continuous function in $\R^d,$ and for all $x\in \R^d$ \begin{equation} f(x)=f(-x),~~~~~
            0<\min\limits_{x\in \R^d}f(x)\leq f(x)\leq \sup_{x\in \R^d}f(x)<+\infty,
            \end{equation}
  which implies that 
     \begin{equation}
           \label{kernel conditions} \int_{\mathbb{R}^d}G(y)e^{ay}dy<+\infty~~\forall \vert a\vert<1,~\text{and}~~\int_{\mathbb{R}^d}G(y)e^{ay}dy=+\infty ~\forall \vert a\vert \geq 1.
        \end{equation}
            
   \item \label{item:53} There exist  positive constants $A$ and $B_1$ such that, for any $x\in \R^d$
    \begin{equation}
    \label{eq:sublinear growth initial conditon}
        u^0_{\varepsilon}(x)\leq -A\vert x\vert +B_1.
    \end{equation}
   \item \label{item:63} There exists a positive constant  $L$ such that for all $\varepsilon>0,$ the  function $u^0_\varepsilon$ is $L$-Lipschitz-continuous
    \item \label{item:73} The sequence of functions $(u^0_\varepsilon)_{\varepsilon>0}$ converges locally uniformly to a continuous function $u^0,$ as $\varepsilon \rightarrow 0$.
\end{enumerate}
\section{Preliminary results}
In this section we give some classical results about \eqref{eq:model}. 
\begin{pro}
   Under Assumptions \ref{item:13}-\ref{item:43}, the equation \eqref{eq:model} has unique solution $n_\varepsilon\in C^1(\mathbb{R}^+\times \mathbb{R}^d).$  Moreover, along a subsequence, $(I^\varepsilon)_\varepsilon$ converges pointwise to a function $I$ that is locally of bounded variation on $\mathbb{R}^+$.
\end{pro}
The proof relies on standard arguments and  is independent of the specific form of the kernel (see, e.g., \cite{BMP}).
\section{Feynman-Kac representation}
This section is devoted to a probabilistic representation of \eqref{eq:model}.
Let $(\Omega,\mathcal{F},\mathbb{P})$ be a probabilistic space, and $N$ be a Poisson random measure on $\mathbb{R}^+\times \mathbb{R}^d$ with intensity measure $dsG(dy).$ Let $X^\varepsilon$ be a continuous random walk with jump rate $\frac{1}{\varepsilon},$ and i.i.d jump sizes distributed as $\varepsilon Z,$ where the random variable $Z$ with density $G(.).$ In other words
\begin{equation}
\label{eq:PSDE}
    X^\varepsilon_t=\varepsilon \int_{0}^{t/\varepsilon}\int_{\mathbb{R}^d}yN(ds,dy).
\end{equation}
We denote the process started at $x$ by $X^{\varepsilon,x}.$
\begin{pro}
    Under Assumption \ref{item:43}, the process $X^\varepsilon$ is well-defined, with independent, stationary increments.
\end{pro}
\begin{proof}
The well-posedness until any time $T>0,$ is a direct consequence of \ref{item:43}. In fact that, let $M>0,$ we have 
\begin{align*}
    \mathbb{P}(\sup_{t\in[0,T]}\vert X^\varepsilon_t\vert >M)&\leq \frac{1}{M}\mathbb{E}(\sup_{t\in[0,T]}\vert X^\varepsilon_t\vert)\\
    &\leq \frac{1}{M}\mathbb{E}\Big(\varepsilon\int_{0}^{T/\varepsilon}\int_{\mathbb{R}^d}\vert y\vert N(ds,dy)\Big)\\
    &\leq \frac{T}{M}\int_{\mathbb{R}^d}\vert y\vert G(y)dy.
\end{align*}
Note that from  Assumption \ref{item:43}, we have $\int_{\mathbb{R}^d}\vert y\vert G(y)dy<+\infty.$ We deduce that the process $X^\varepsilon$ is well defined.
\end{proof}
\begin{thm}
    Under Assumptions \ref{item:13}-\ref{item:63}, let $n_\varepsilon$ be the unique solution of \eqref{eq:model}, then
    \begin{equation}
        \label{eq:Fynman-Kac}
n_\varepsilon(t,x)=\mathbb{E}_x(e^{\frac{1}{\varepsilon}(u^0_\varepsilon(X^\varepsilon_t)+\int_{0}^{t}R(X^\varepsilon_s,I_\varepsilon(t-s))ds)}),~~\forall (t,x)\in (0,+\infty)\times \mathbb{R}^d.
    \end{equation} 
\end{thm}
  \begin{proof}
      Let $(t,x)\in(0,+\infty)\times \mathbb{R}^d$ fixed. Let us consider  the stochastic process defined as follows:
      \begin{equation}
    M^{t,\varepsilon}_s=e^{\frac{1}{\varepsilon}\int_{0}^{s}R(X^{\varepsilon,x}_\tau,I_\varepsilon(t-\tau))d\tau}n_\varepsilon(t-s,X^{\varepsilon,x}_s),~~\forall s\leq t.
      \end{equation}
      We prove that $s\mapsto M^{t,\varepsilon}_s$ is a martingale. Under Assumptions \ref{item:13}-\ref{item:63}, $M^{t,\varepsilon}$ is bounded (see Lemma 4.1 in \cite{J}, for a bound of $n_\varepsilon$), and by Ito's lemma and \eqref{eq:model}, we obtain that
\begin{equation}
\begin{aligned}
M^{t,\varepsilon}_s&=n_\varepsilon(t,x)+  \int_{0}^{s}\Bigg(\frac{1}{\varepsilon}R(X^{\varepsilon,x}_\tau,I^\varepsilon(t-\tau))M^{t,\varepsilon}_\tau-e^{\frac{1}{\varepsilon}\int_{0}^{\tau}R(X^{\varepsilon,x}_u,I^\varepsilon(t-u))du}\partial_tn_\varepsilon(t-\tau,X^{\varepsilon,x}_\tau)\Big)d\tau\\
&+\int_{0}^{s}\int_{\mathbb{R}^d}e^{\frac{1}{\varepsilon}\int_{0}^{\tau}R(X^{\varepsilon,x}_u,I^\varepsilon(t-u))du}(n_\varepsilon(t-\tau,X^{\varepsilon,x}_{\tau{^-}}+\varepsilon y)-n_\varepsilon(t-\tau,X^{\varepsilon,x}_{\tau{^-}}))N(d\tau,dy)\\
&=n_\varepsilon(t,x)+\int_{0}^{s}\int_{\mathbb{R}^d}e^{\frac{1}{\varepsilon}\int_{0}^{\tau}R(X^{\varepsilon,x}_u,I^\varepsilon(t-u))du}(n_\varepsilon(t-\tau,X^{\varepsilon,x}_{\tau{^-}}+\varepsilon y)-n_\varepsilon(t-\tau,X^{\varepsilon,x}_{\tau{^-}}))\widetilde{N}(d\tau,dy),
 \end{aligned}
\end{equation}
where $\widetilde{N}$ is the compensated Poisson random measure. Hence $M^{t,\varepsilon}$ is a martingale. Using the fact that $\mathbb{E}(M^{t,\varepsilon}_t)=\mathbb{E}(M^{t,\varepsilon}_0),$ we obtain the Feynman-Kac formula \eqref{eq:Fynman-Kac}.    
\end{proof}      
\section{Varadhan's Lemma}
We start this section by stating the large deviation principle.
\begin{thm}
    Under Assumptions \ref{item:43}. The family of laws of the process $(X^\varepsilon)_{\varepsilon>0}$ satisfies an extended large deviation principle on $\mathbb{V}$ endowed with the Borovkov metric $\rho_B$, with the rate function, for all $t>0,$ and $y\in \mathbb{V}_t,$
    \begin{equation}
        \label{eq:LDP rate}
        J_t(y)=
        \begin{cases}
\int_{0}^{t}L(y'_s)ds,~~\text{where $y\in AC([0,t],\mathbb{R}^d)$, }\\
+\infty,~~~~~\text{otherwhere},
        \end{cases}
    \end{equation}
    where $L$ denotes the Legendre transform on $\mathbb{R}^d$ of $H(p)=\int_{\mathbb{R}^d}G(z)(e^{p.z}-1)dz,$ i.e 
    \begin{equation}
    \label{eq:leegendre traansform}
        L(q)=\sup_{p\in \mathbb{R}^d}\{p.q-H(p)\},
        \end{equation}
        and $AC([0,t],\mathbb{R}^d)$ denotes the set of absolutely continuous functions on $[0,t]$.
\end{thm}
The proof of this result is given in Theorem 1.1 in \cite{M}.

\begin{lem} Under assumptions \ref{item:13}-\ref{item:73}
   For all $T>0,$ the functional 
\begin{equation}
    \Phi_\varepsilon:\mathbb{V}_T\to \mathbb{R},~
    y\mapsto \int_{0}^{T}R(y_s,I_{T-s}^\varepsilon)ds+u^0_\varepsilon(y_T),
\end{equation}
is continuous on $\mathbb{V}_T$ endowed with the Borovkov metric $\rho_B$. Moreover, we have along a subsequence of $\varepsilon$ that, the functional $\Phi_{\varepsilon_k}$ converge locally uniformly in $(\mathbb{V}_T,\rho_B)$ to the continuous functional $\Phi$ defined  for all $y\in \mathbb{V}_T$ by  

\begin{equation}
\Phi(y)=:\int_{0}^{T}R(y_s,I_{T-s})ds+u^0(y_T),
\end{equation}
where $I$ is the BV function given in Proposition 2.1.  
\end{lem}

\begin{proof}
Fix $\varepsilon>0.$
Let us now prove the continuity of $\Phi.$  Let $(y^n)_{n\ge 1}$ and $y$ be elements of $\mathbb{V}_T$ such that
\[
\rho_{B}(y^n, y) \longrightarrow 0 \quad \text{as } n \to \infty.
\]
Since $y$ is of bounded variation, it has at most a countable number of jumps, and we have, for all points $s$ of continuity of $y$, that
\begin{equation}
    y^n_s\rightarrow y_s,~~\text{as}~ n\rightarrow +\infty,
\end{equation}
Furthermore, $x\mapsto R(x,.)$ is continuous and bounded, then by dominated convergence theorem we have 
\begin{equation}
    \int_{0}^{T}R(y^n_s,I^\varepsilon_{T-s})ds\to  \int_{0}^{T}R(y_s,I^\varepsilon_{T-s})ds,~~\text{as}~n\rightarrow +\infty.
\end{equation}
Let us prove now the continuity of the second part of the functional $\Phi.$ Let $T'>T.$ Since $y$ is constant in $(T,+\infty),$ then $y$ is continuous at $T'.$ Moreover, the convergence in the Borovkov metric implies that
\begin{equation*}
    y^n_{T'}\to y_{T'}, ~~\text{as}~n\rightarrow +\infty.
    \end{equation*}
     Since $T'>T$ and $y^n,y\in \mathbb{V}_T.$ Thus 
        \begin{equation}
            y^n_{T}= y^n_{T'}\to y_{T'}= y_{T}, ~~\text{as}~n\rightarrow +\infty.
 \end{equation}
 Finally $\Phi$ is continuous in $(\mathbb{V}_T,\rho_B).$\\
For each $y\in \mathbb{V}_T,$ the convergence of $\Phi_{\varepsilon_k}(y)$ to $\Phi(y)$ is given by the convergence of $I_{\varepsilon_k}$ and the assumption \ref{item:73}. Moreover, since $I\mapsto R(.,I)$ is Lipschitz-continuous, then the integral term converges globally in $(\mathbb{V}_T,\rho_B).$ Moreover  $u^0_\varepsilon $ convergences locally uniformly and $y\mapsto y_T$ is continuous with $\rho_B.$ Therefore the convergence is locally uniformly in $(\mathbb{V}_T,\rho_B).$

\end{proof}


We now state our main result, which extends Varadhan's lemma.
\begin{thm}
    Under Assumptions \ref{item:13}- \ref{item:73}, for all $(t,x)\in (0,+\infty)\times \mathbb{R}^d,$ we have 
    \begin{equation}
    \label{eq:variationnal formulation}
        \lim\limits_{k\rightarrow +\infty} u_{\varepsilon_k}(t,x)=u(t,x):=\sup\limits_{y \in AC([0,t],\mathbb{R}^d),y_0=x}\{u^0(y_t)+\int_{0}^{t}R(y_s,I_{t-s})ds-\int_{0}^{t}L(y'_s)ds\}.
    \end{equation}
\end{thm}
This theorem provides a variational representation and reveals a surprising effect: the variational representation retains the classical form even though the equation involves an obstacle. This is explained by the fact that the Lagrangian is always finite, even though the Hamiltonian can be infinite. The variational representation encodes the obstacle through the fact that the Lagrangian is forced to be $
1$-Lipschitz, a property that follows from the Hamiltonian being finite only on the unit ball. This will appear in detail in the proof of the link between the variational representation and the Hamilton–Jacobi equation \eqref{eq:HJ3} in Section 5 below.
\begin{proof}
Let $T>0,$ and $y\in \mathbb{V}_T$ such that $y_0=x.$ The functional $\Phi$ is continuous at $y$ in $\mathbb{V}_T$ endowed with $\rho_B.$ Then, for  all $\delta>0$ there exists a neighborhood $V_y^{\delta}$ of $y$ on $(\mathbb{V}_T,\rho_B)$ such that 
 \begin{equation} 
      \Phi(y)-\delta \leq \inf_{z\in V^{\delta}_y}\Phi(z).
 \end{equation}
Moreover, the local uniform convergence of $\Phi_{\varepsilon_k}$ entails that
 \begin{equation} 
 \label{eq:ineq for the convergence}
     \Phi(y)-2\delta\leq \inf_{z\in V^\delta_y}\Phi_{\varepsilon_k}(z).
 \end{equation}
 Then
 
\begin{align*}
\varepsilon_k \log \mathbb{E}\left(
e^{\frac{1}{\varepsilon_k}\Phi_{\varepsilon_k}(X^{\varepsilon_k,x})}
\right)
&\geq
\varepsilon_k \log \mathbb{E}\left(
e^{\frac{1}{\varepsilon_k}\Phi_{\varepsilon_k}(X^{\varepsilon_k,x})}
\mathbf{1}_{\{X^{\varepsilon_k,x}\in V_y^\delta\}}
\right) \\
&\geq
\Phi(y)-2\delta
+
\varepsilon_k \log
\mathbb{P}\left(
X^{\varepsilon_k,x}\in V_y^\delta
\right).
\end{align*}
 
 From lower bound of the LDP stated in Theorem 4.1, we deduce that  \begin{equation}
     \liminf_{k\rightarrow +\infty}\varepsilon_k\log{\mathbb{E}\big(e^{\frac{1}{\varepsilon_k}\Phi_{\varepsilon_k}(X^{\varepsilon_k,x})}\big)}\geq \Phi(y)-2\delta-\inf_{z\in V^\delta_y}J_t(z)\geq \Phi(y)-2\delta-J(y).
 \end{equation}
 Then, by letting $\delta$ go to zero, we obtain  that 
 \begin{equation}
      \liminf_{k\rightarrow +\infty}\varepsilon_k\log{\mathbb{E}\big(e^{\frac{1}{\varepsilon_k}\Phi_{\varepsilon_k}(X^{\varepsilon_k,x})}\big)}\geq \sup_{y\in AC([0,t],\R^d), y_0=x}\{\Phi(y)-J_t(y)\}.
 \end{equation}
 Let us now prove the upper bound. To simplify the notations we denote the process $X^{\varepsilon_k,x}$ as $X^\varepsilon.$\\
 We start by defining the functional $\widetilde\Phi$ in $\widetilde{\mathbb{V}}_T$ by $\widetilde\Phi(\widetilde{y})=\widetilde\Phi((y^+,y^-))=\Phi(y_+-y_-).$
From \eqref{eq:metric comparaison} the functional $\widetilde\Phi$ is continuous in $\widetilde{\mathbb{V}}_T$ endowed with the Lévy metric $\widetilde{\rho}_L.$ 
Fix $\tilde f\in\widetilde{\mathbb{V}}_T$ such that $f_0=x$. Let $\eta>0.$
 Since $\widetilde{\Phi}_\varepsilon$ converges  locally uniformly, and $\widetilde{\Phi}$ is continuous, then there exists
$\delta_1=\delta_1(\tilde f,\eta)>0$ such that for all $\tilde g\in B_L({\tilde f},\delta_1)$ we have 
\begin{equation}
\big|\widetilde{\Phi}_\varepsilon(\tilde g)-\widetilde{\Phi}_\varepsilon(\tilde f)\big|<\eta.
\end{equation}
Moreover, by combining \eqref{eq:lower ldp for app pr} with \eqref{eq:expo app} we get that
$$\displaystyle\lim_{\delta\downarrow0}\limsup_{\varepsilon \rightarrow 0}\log
P(\widetilde{X}^\varepsilon \in B_L({\tilde f},{\delta}))\le-J(\tilde f).$$
Thus, there exists $\delta_0=\delta_0(\tilde f,\eta)>0$ such that, for every
$\delta\in(0,\delta_0]$,
\begin{equation}
\label{eq:proba continuity}
\limsup_{\varepsilon \to 0}
\varepsilon \log P\big(\widetilde{X}^\varepsilon \in B_L(\tilde f, \delta)\big)
\le -J(\tilde f) + \eta.
\end{equation}
Set
$
\delta(\tilde f):=\min\{\delta_0(\tilde f,\eta),\,\delta_1(\tilde f,\eta)\}>0.
$
From \eqref{eq:expontighntess apro} and \eqref{eq:expo app}, we deduce that for all $M$ there exists a compact $\mathcal K_M \subset  \widetilde{\mathbb{V}}_T$ endowed with $\widetilde{\rho}_L$ such that 
\begin{equation}
\label{eq:expontighntess}
    \limsup_{\varepsilon\to 0}\varepsilon\log{\mathbb{P}(\widetilde{X}^\varepsilon\in \mathcal K_M^c)}\leq -M.
\end{equation}
The family $\left( B_L(\tilde f, \delta(\tilde f)) \right)_{\tilde f \in \mathcal K_M}$ is an open
cover of $\mathcal K_M$ with respect to $\widetilde{\rho}_L$. Since $\mathcal K_M$ is compact, there is a
finite subcover:
\begin{equation}
\mathcal K_M\ \subset\ \bigcup_{i=1}^{m}B_L({\tilde f_i},\delta_i),
\qquad \delta_i:=\delta(\tilde f_i),\ i=1,\dots,m,
\end{equation}
where $m\in \mathbb{N}^*$. Therefore, for each $i,$ we have
\begin{align}
\mathbb E\Big[\exp\big(\frac{1}{\varepsilon}\widetilde{\Phi}(\widetilde{X}^\varepsilon)\big)\,
\mathbf 1_{ B_L(\tilde{f}_i, \delta_i)}(\widetilde{X}^\varepsilon)\Big]
\ \le\ \exp\big(\frac{1}{\varepsilon}(\widetilde{\Phi}(\tilde f_i)+\eta)\big)\cdot
P\big(\widetilde{X}^\varepsilon\in  B_L(\tilde{f}, \delta)\big).
\end{align}
Using \eqref{eq:proba continuity} we conclude that 
\begin{equation}
    \limsup_{\varepsilon\rightarrow 0}\varepsilon\log{\mathbb{E}\Big[\exp\big(\frac{1}{\varepsilon}\widetilde{\Phi}_\varepsilon(\widetilde{X}^\varepsilon)\big)\,
\mathbf 1_{ B_L(\tilde{f}_i, \delta_i)}\Big]}\leq \widetilde{\Phi}(\tilde{f}_i)+2\eta -J_t(\tilde{f_i}),
\end{equation}
From Lemma 1.2.15 in \cite{ZD}, we obtain that
\begin{align*}
\limsup_{\varepsilon \to\infty}\varepsilon \log
\mathbb E\big[\exp(\frac{1}{\varepsilon}\widetilde{\Phi}(\widetilde{X}^\varepsilon))\,\mathbf 1_{\mathcal K_M}(\widetilde{X}^\varepsilon)\big]
&\ \le\ \max_{1\le i\le m}\big\{\widetilde{\Phi}(\tilde f_i)-J_t(\tilde f_i)\big\}+2\eta\\
&\ \le\ \sup_{\tilde f\in\widetilde{\mathbb{V}}_t,f_0=x}\{\widetilde{\Phi}(\tilde f)-J_t(\tilde f)\}+2\eta.
\end{align*}
On the other hand, using \ref{item:23} and \ref{item:53},  we get that
\begin{equation}
    \limsup_{\varepsilon \to\infty}\varepsilon \log
\mathbb E\big[\exp(\frac{1}{\varepsilon}\widetilde{\Phi}(\widetilde{X}^\varepsilon))\,\mathbf 1_{\mathcal K_M^c}(\widetilde{X}^\varepsilon)\big]\leq \overline{R}t+B_1-M.
\end{equation}
Finally, we deduce that 

\begin{equation}
\limsup_{\varepsilon \to\infty}\varepsilon \log
\mathbb E\big[\exp(\frac{1}{\varepsilon}\widetilde{
\Phi}(\widetilde{X}^\varepsilon))\,\mathbf )\big]
\ \le\ \max\{\sup_{\tilde f\in\widetilde{\mathbb{V}}_t,f_0=x}\{\widetilde{\Phi}(\tilde f)-J_t(\tilde f)\}+2\eta,\overline{R}t+B_1-M\}.
\end{equation}
By letting $M$ go to infinity and $\eta$ to zero we deduce the upper bound. Moreover, Corollary 3.2 in \cite{M}, implies that 
\begin{equation}
    \limsup_{\varepsilon \to\infty}\varepsilon \log
\mathbb E\big[\exp(\frac{1}{\varepsilon}\Phi(\widetilde{X}^\varepsilon)\,\mathbf )\big]
\ \le\ \sup_{ f\in\mathbb{V}_t, f_0=x}\{\Phi(f)-J_t(f)\}.
\end{equation}
\end{proof}

\section{Variational representation of the viscosity solution}
    
\begin{thm}
 Under Assumptions \ref{item:13}-\ref{item:73},the function $u$ given by \eqref{eq:variationnal formulation} is a viscosity  solution to the constrained Hamilton-Jacobi equation with obstacle
 \eqref{eq:HJ3}.
 \end{thm}
Before proving this theorem, we establish a dynamic programming principle that will be useful in the proof. 
 \begin{pro}
     The variational representation in \eqref{eq:variationnal formulation} satisfies a  dynamic programming principle formula, (i,e), for all $t>s\geq 0,$ and $x\in \mathbb{R}^d,$ we have that 
     \begin{equation}
     \label{eq:DPP}
         u(t,x)=\sup\limits_{y\in AC([0,t],\mathbb{R}^d),y_t=x}\{ u(s,y_{s})+\int_{s}^{t}R(y_\tau,I_\tau)d\tau-\int_{s}^{t}L(-y'_\tau) d\tau\}.
     \end{equation}
 \end{pro}
\begin{proof}
The proof follows classical arguments in dynamic programming principle of optimal control (see, for instance, \cite{BD,cannarsa2004}).\\
Let $t>s>0,$  $x\in \mathbb{R}^d.$ It is clear from   \eqref{eq:variationnal formulation} that

\begin{equation}
\label{eq:newrepresentation dpp}
     u(t,x)=\sup\limits_{y \in AC([0,t],\mathbb{R}^d),y_t=x}\{u^0(y_0)+\int_{0}^{t}R(y_\tau,I_{\tau})d\tau-\int_{0}^{t}L(-y'_\tau)d\tau\}.
\end{equation}
Let $y\in AC([0,t],\mathbb{R}^d)$ such that $y_t=x.$  We have for all $z\in AC([0,s],\R^d)$ such that $z_s=y_{s},$ that

$$u(s,y_{s})\geq u^0(z_0)+\int_{0}^{s} R(z_\tau,I_\tau)d\tau-\int_{0}^{s}L(-z'_\tau) d\tau.$$
We take 
$z_\tau=y_{\tau}, \forall \tau\in[0,s]$. Thus 
$$u(s,y_{s})+ \int_{s}^{t} R(y_\tau,I_\tau)d\tau-\int_{s}^{t}L(-y'_\tau) d\tau\geq u^0(y_0)+\int_{0}^{t} R(y_\tau,I_\tau)d\tau-\int_{0}^{t}L(-y'_\tau) d\tau.$$
 Therefore, by taking the supremum overall $y\in AC([0,t],\R)$ such that $y_t=x$, we get that 
  $$ \sup\limits_{y\in AC([0,t],\mathbb{R}^d), y_0=x}\{ u(s,y_{s})+\int_{s}^{t}R(y_\tau,I_\tau)-\int_{s}^{t}L(-y'_\tau) d\tau\}\geq u(t,x).$$
  Let us now prove the converse inequality: let $y\in AC([0,t],\mathbb{R}^d),$ such that $y_t=x.$ Let $\eta>0,$ from \eqref{eq:newrepresentation dpp} there exist $y_\eta\in AC([0,s],\mathbb{R}^d), $ such that $y_\eta(s)=y_{s},$ and
  \begin{equation}
  \label{eq:a}
    u(s,y_{s})\leq u^0(y_\eta(0))+\int_{0}^{s} R(y_\eta(\tau),I_\tau)d\tau-\int_{0}^{s}L(-y'_\eta(\tau))d\tau+\eta.
    \end{equation}
  We define $$z_\tau=\begin{cases}
   y_\tau,~~~~~~~\tau \in [0,s],\\
      y_\eta(\tau),~\tau \in [s,t].
  \end{cases}$$
  From \eqref{eq:newrepresentation dpp} and \eqref{eq:a}, we have 
  \begin{align*}
      u(t,x)&\geq u^0(z_0)+\int_{0}^{t}R(z_\tau,I_\tau)-\int_{0}^{t}L(-z'_\tau) d\tau\\
      &\geq u(s,y_{s})+\int_{s}^{t}R(y_\tau,I_\tau)-\int_{s}^{t}L(-y'_\tau) d\tau-\eta.
  \end{align*}
  By taking the supremum, and letting $\eta$ go to zero, we conclude the proof. 
\end{proof}
\subsection{Proof of Theorem 5.1}
We start by proving the following lemma 
\begin{lem}
  Under Assumptions \ref{item:13}-\ref{item:73}, The function $u$ is Lipschitz-continuous in time and space, and is $1$-Lipschitz with respect to the spatial variable.
  \end{lem}
  \begin{proof}
      Let us first proof that $u$ is 1-Lipschitz w.s.t space. The proof of this propriety relies on the fact that $L$ is $1$-Lipschitz, i,e

      \begin{equation}
      \label{eq:Lipschitz lagrangien}
          L(p)-L(q)\leq \vert p-q\vert,~~\forall p,q\in \R^d.
      \end{equation}
      The proof of this point will be given after the proof of this Lemma.\\
Let $t>s>0,$  $x,y\in \R^d$ and $\eta>0.$ From \eqref{eq:DPP} there exists $z^\eta\in AC([0,t],\R^d)$ such that $z^\eta_t=x,$ and \begin{equation}
       \label{eq:updpp}
           u(t,x)\leq  u(s,z^\eta_s)+\int_{s}^{t}R(z^\eta_\tau,I_\tau)d\tau-\int_{s}^{t}L(-(z^\eta_\tau)'))d\tau+\eta.
       \end{equation}
      Moreover, for all $z\in AC([0,t],\R^d),$ such that $z_t=x',$ we have
      \begin{equation}
      \label{eq:lowerdpp}
          u(t,x')\geq u(s,z_s)+\int_{s}^{t} R(z_\tau,I_\tau)d\tau-\int_{s}^{t}L(-z'_\tau)d\tau.
      \end{equation}
We define $\widetilde{z}^\eta_\tau=z^\eta_\tau+\frac{\tau-s}{t-s}(x-x'),~\forall \tau\in [s,t].$ By \eqref{eq:updpp} and \eqref{eq:lowerdpp} we obtain that

\begin{align*}
    u(t,x)-u(t,x')&\leq \int_{s}^{t}(R(z^\eta_\tau,I_\tau)-R(\widetilde{z}^\eta_\tau,I_\tau))d\tau+ \int_{s}^{t} L(-(z^\eta_\tau)')-L\big(-(z^\eta_\tau)'+\frac{x-x'}{t-s}\big)d\tau+\eta\\
    &\leq \Vert R\Vert_{Lip} \int_{s}^{t}\frac{\tau-s}{t-s}\vert x-x'\vert d\tau+\int_{s}^{t}\frac{\vert x-x'\vert}{t-s}d\tau+\eta\\
    &\leq \Vert R\Vert_{Lip} \frac{(t-s)}{2}\vert x-x'\vert +\vert x-x'\vert+\eta.
\end{align*}
By letting $\eta$  go to zero and $s$ to $t$ we conclude the proof.\\

Let us now prove \eqref{eq:Lipschitz lagrangien}. Note that  $H(p)=+\infty$ for all $\vert p\vert \geq 1,$ and $H$ is positive then for all $p\in \R^d$ we have 
\begin{equation}
    L(p)=\sup_{q\in \R^d}\{p.q-H(q)\}=\sup_{q\in B(0,1)}\{p.q-H(q)\}\leq \vert p\vert.
\end{equation}
Furthermore, by the convexity of $L$ there exists $\xi\in \partial L(p)$ such that 
\begin{equation}
\label{eq:subdiff}
    L(q)\geq L(p)+\xi.(q-p),~~\forall q\in \R^d.
\end{equation}
We take $q=p+\lambda.v$ where $v\in \R^d$ and $\lambda \in \R.$ Then \eqref{eq:subdiff} implies that 
\begin{align*}
    \xi.v\leq\frac{L(p+\lambda v)-L(P)}{\lambda}\leq \frac{\vert p+\lambda v\vert-L(p) }{\lambda}.
\end{align*}
We let $\lambda$ go to infinity we obtain that 
\begin{equation*}
    \xi.v\leq \vert v\vert, ~~\forall v\in \R^d.
\end{equation*}
    Therefore, 
    \begin{equation}
        \vert \xi \vert\leq 1.
    \end{equation}
Finally, we conclude from \eqref{eq:subdiff} that 
\begin{equation}
    L(q)-L(p)\geq -\vert p-q\vert,
\end{equation}
which ensures that $L$ is $1$-Lipschitz.\\

We prove now that
\[
|u(t,x) - u(t',x)| \leq C |t - t'|.
\]
The argument is similar to the proof of Lipschitz continuity in space. We compare the two value functions by transporting an almost optimal trajectory from time $t$ to time $t'$. Without loss of generality, assume that $t' > t$.
Let $\eta>0,$ there exists a trajectory $y \in AC([0,t], \mathbb{R}^d)$ with $y_t = x$ such that
\[
u(t,x) \leq u(s,y_s) +\int_{s}^{t}R(y_\tau,I_\tau)d\tau- \int_s^t L(-y'_\tau)\,d\tau + \eta.
\]
We extend $y$ on $[t,t']$ by setting
\[
\tilde{y}_\tau = x \quad \text{for } \tau \in [t,t'].
\]
Then $\tilde{y}'_\tau = 0$ on $[t,t']$, and $\tilde{y}_{t'} = x$, $\tilde{y}_s = y_s,~~\forall s\in [0,t]$.
By the variational representation of $u(t',x)$, we obtain
\begin{align*}
u(t',x) &\geq u(s,\tilde{y}_s)+\int_{s}^{t'}R(\tilde{y}_\tau,I_\tau)d\tau - \int_s^{t'} L(-\tilde{y}'_\tau)\,d\tau\\
&\geq u(s,y_s)+\int_{s}^{t}R(y_\tau,I_\tau)d\tau+\int_{t}^{t'}R(x,I_\tau)d\tau- \int_{s}^{t} L(-y'_\tau)\,d\tau - \int_t^{t'} L(0)\,d\tau.
\end{align*}
Therefore, from assumption \ref{item:23}, we obtain that
\[
u(t',x) - u(t,x) \geq -(\overline{R}+L(0))(t' - t) - \eta.
\]
We now take an almost optimal trajectory $z$ for $u(t',x)$ with $z_{t'} = x$, and restrict it to $[s,t]$. We obtain
\[
u(t,x) \geq u(s,z_s) +\int_{s}^{t}R(z_\tau,I_\tau)d\tau- \int_s^t L(-{z}'_\tau)\,d\tau.
\]
Hence,
\[
u(t',x) - u(t,x) \leq \int_{t}^{t}R(z_\tau,I_\tau)d\tau - \int_t^{t'} L(-{z}'_\tau)\,d\tau + \eta
\leq \overline{R} (t'-t)\int_t^{t'} |{z}'_\tau|\,d\tau + \eta,
\]
where we used that $L(p) \leq |p|$.\\
Moreover, we have that $L(p) \to +\infty$ as $|p| \to +\infty$. It is easy to see, for instance by choosing
\[
q = \frac{1}{\sqrt{d+1}}(\operatorname{sign}(p_i))_{1 \le i \le d},
\]
in the Legendre transform, that
\[
L(p) \ge C_1 |p| - C_2,
\]
for some positive constants $C_1, C_2$.
In particular, this implies that  there exists $M > 0$ such that the supremum in the variational representation is attained over trajectories satisfying
\[
\|y'\|_{L^\infty([0,t])} \le M.
\]
Thus
\begin{equation}
u(t',x) - u(t,x) \leq (M+\overline{R}) |t' - t|.
\end{equation}
Since $L$ is the Legendre transform of $H$, and $H$ is positive, it follows that $L(0)$ is negative. Therefore 
\begin{equation}
\vert u(t',x) - u(t,x)\vert \leq (M+\overline{R}) |t' - t|.
\end{equation}
\end{proof}
  
\subsubsection{Proof of Theorem 5.1}

Let us  prove that $u$ is a supersolution of \eqref{eq:HJ3}.  Let $(t,x)\in (0,+\infty)\times \mathbb{R}^d,$ and $\phi \in C^1((0,+\infty)\times \mathbb{R}^d),$ such that $u-\phi$ has a minimum at $(t,x).$ Let $s<t$, we have for all $y\in AC([0,t],\mathbb{R}^d))$ such that $y_t=x,$ and $y'=p$ (a,e) in $[0,t],$ with $p\in \mathbb{R}^d,$ that 
\begin{equation}
  u(t,x)\geq u(s,y_{s})+\int_{s}^{t}R(y_\tau,I_\tau)d\tau-\int_{s}^{t}L(-p)d\tau. 
\end{equation}
Note that $(t,x)$ is a minimum point of $u-\phi,$ then for $s$ close enough to $t,$ we obtain 
\begin{equation}
    u(t,x)-\phi(t,x))\leq u(s,y_s)-\phi(s,y_s).
\end{equation}
This entails
\begin{equation*}
    \phi(t,x)-\phi(s,y_s)\geq \int_{s}^{t}R(y_\tau,I_\tau)d\tau-\int_{s}^{t}L(-p)d\tau.
\end{equation*}
Thus,
\begin{equation*}
\int_{s}^{t}\partial_\tau(\phi(\tau,y_\tau)d\tau \geq \int_{s}^{t}R(y_\tau,I_\tau)d\tau-\int_{s}^{t}L(-p)d\tau.
\end{equation*}
Therefore, 
\begin{equation}
\int_{s}^{t}\partial_t\phi(\tau,y_\tau)d\tau \geq \int_{s}^{t}R(y_\tau,I_\tau)d\tau+\int_{s}^{t}(-y'_\tau\partial_x\phi(\tau,y_\tau)-L(-p))d\tau.
\end{equation}
We divide by $t-s,$ and letting $s$ go to $t,$ we get
\begin{equation}
    \partial_t\phi(t,x)\geq \liminf_{s\to t}R(x,I_s)+-p.\nabla \phi(t,x)-L(-p).
\end{equation}
Since $p$ is arbitrary, we conclude that 
\begin{equation}
     \partial_t\phi(t,x)\geq \liminf_{s\rightarrow t}R(x,I_s)+H(\nabla \phi(t,x)).
\end{equation}
Let us now  prove that \begin{equation}
\label{eq:nabla}
    \vert \nabla \phi(t,x)\vert \leq 1.\end{equation}
Since $(t,x)$ is a minium point of $u-\phi,$ and using the fact that $u$ is 1-Lipschitz, we obtain that for all $h\in \R$ we have 
\begin{equation}
    \phi(t,x+h)-\phi(t,x)\leq u(t,x+h)-u(t,x)\leq \vert h\vert,
\end{equation}
this implies \eqref{eq:nabla}. 

The proof that $u$ is a subsolution follows by similar arguments. The constraint $\max_{x\in \mathbb{R}^d} u(t,x) = 0$ is independent of the variational formulation and is proved exactly as in \cite{BMP,J}.
\newpage
\textbf{Acknowledgments}: The author thanks Nicolas Champagnat, Sylvie Méléard and Sepideh Mirrahimi for fruitful discussions and comments. This work is funded by the European Union (ERC,
SINGER, 101054787). Views and opinions expressed are however those of the author
only and do not necessarily reflect those of the European Union or the European Research
Council. Neither the European Union nor the granting authority can be held responsible
for them.

\newpage
\nocite{*}
\bibliographystyle{plain}
\bibliography{ref.bib}
\end{document}